\documentclass[10pt]{amsart} 
\usepackage[foot]{amsaddr}
\usepackage{amsmath,amssymb,bm, color}

\usepackage{amsmath}
\usepackage{graphicx,float} 

\usepackage{mathrsfs}
\usepackage{bbm}
\usepackage{bm}
 
\usepackage{amsfonts,amssymb}
 
\usepackage{multirow}
\usepackage{lineno}
 
\usepackage{color}

\numberwithin{equation}{section}
 
 \newtheorem{theorem}{Theorem}[section]
 \newtheorem{lemma}[theorem]{Lemma}

 \def\3bar{{|\hspace{-.02in}|\hspace{-.02in}|}}

\def\T{{\mathcal{T}}}

\def\beta{\boldsymbol{\eta}}

\def\cal#1{{\mathcal #1}}
\def\pT{{\partial T}}

\def\bw{{\mathbf{w}}}
\def\bf{{\mathbf{f}}}
\def\bg{{\mathbf{g}}}
\def\bu{{\mathbf{u}}}
\def\bv{{\mathbf{v}}}
\def\bn{{\mathbf{n}}}

\def\bvarphi{{\boldsymbol{\varphi}}}

\newtheorem{algorithm}{WG Algorithm}[section]

\numberwithin{equation}{section}

\def\3bar{{|\hspace{-.02in}|\hspace{-.02in}|}}
\def\p#1{\begin{pmatrix}#1\end{pmatrix}}
 \def\cal#1{\mathcal{#1}}
\def\ad#1{\begin{aligned}#1\end{aligned}}  \def\b#1{\mathbf{#1}} 
\def\a#1{\begin{align*}#1\end{align*}} \def\an#1{\begin{align}#1\end{align}}

\begin{document}

\title []
 {  Weak Galerkin Methods  for the Brinkman Equations}

  \author {Chunmei Wang$\dag$}\thanks{$\dag$ Corresponding author. } 
  \address{Department of Mathematics, University of Florida, Gainesville, FL 32611, USA. }
  \email{chunmei.wang@ufl.edu}
  \thanks{The research of Chunmei Wang was partially supported by National Science Foundation Grant DMS-2136380.}

\author {Shangyou Zhang}
\address{Department of Mathematical Sciences,  University of Delaware, Newark, DE 19716, USA}   \email{szhang@udel.edu}

\begin{abstract}
 This paper introduces a novel weak Galerkin (WG) finite element method for the numerical solution of the Brinkman equations. The Brinkman model, which seamlessly integrates characteristics of both the Stokes and Darcy equations, is employed to describe fluid flow in multiphysics contexts, particularly within heterogeneous porous media exhibiting spatially variable permeability. The proposed WG method offers a unified and robust approach capable of accurately capturing both Stokes- and Darcy-dominated regimes. A discrete inf-sup condition is established, and optimal-order error estimates are rigorously proven for the WG finite element solutions. Furthermore, a series of numerical experiments is performed to corroborate the theoretical analysis, demonstrating the method's accuracy and stability  in addressing the complexities inherent in the Brinkman equations.
 
\end{abstract}


 
 \maketitle
 
\section{Introduction}  
This work focuses on the design and analysis of stable and efficient numerical schemes for the Brinkman equations based on the weak Galerkin (WG) finite element methodology. The Brinkman system provides a unified framework for modeling fluid flow in heterogeneous porous media, where the permeability exhibits significant spatial variability. Such variations naturally give rise to zones governed primarily by Darcy’s law, while others are better described by Stokes flow. In a simplified form, the Brinkman model seeks to compute the velocity field  $\bu$ and the pressure field $p$ by solving the following system:
\begin{equation}\label{model}
 \begin{split}
  -\mu \Delta \bu+\nabla p+\mu \kappa^{-1} \bu=&\bf, \qquad\qquad \text{in}\quad \Omega,\\
 \nabla\cdot  \bu=&0,\qquad\qquad \text{in}\quad \Omega,\\
 \bu=&\bg,\qquad\qquad \text{on}\quad \partial\Omega.
 \end{split}
 \end{equation}Here, $\mu$ denotes the dynamic viscosity of the fluid, and $\kappa$ is the permeability tensor associated with the porous medium. The computational domain $\Omega \subset \mathbb{R}^d$ is assumed to be polygonal or polyhedral, with spatial dimension $d = 2$ or $3$. The vector field $\mathbf{f}$ represents a prescribed source term corresponding to external momentum input. For the purposes of analysis, and without loss of generality, we consider the special case where $\mathbf{g} = 0$ and $\mu = 1$.
The permeability tensor $\kappa$ is assumed to be symmetric and uniformly positive definite. That is, there exist positive constants $\lambda_1$ and  $\lambda_2 > 0$ such that for any $\xi\in \mathbb R^d$,
 $$
 \lambda_1\xi^T\xi\leq \xi^T \kappa^{-1}\xi\leq \lambda_2\xi^T\xi. 
 $$
 For analytical simplicity, the permeability is taken to be constant throughout this paper. Nonetheless, the theoretical framework developed herein can be naturally extended to handle spatially varying permeability without introducing substantial complications.

The variational formulation corresponding to the Brinkman system \eqref{model} is given as follows: Determine $\bu \in [H_0^1(\Omega)]^d$ and $p \in L_0^2(\Omega)$ such that
\begin{equation}\label{weak}
 \begin{split}
     (\nabla \bu, \nabla \bv)-(\nabla \cdot\bv, p)+(\kappa^{-1} \bu, \bv)=&(\bf, \bv), \qquad \forall \bv\in [H_0^1(\Omega)]^d \\
     (\nabla \cdot\bu, q)=&0, \qquad \qquad\forall q\in L_0^2(\Omega),
 \end{split}
 \end{equation} 
  where $$H_0^1(\Omega)=\{w\in H^1(\Omega): w|_{\partial\Omega}=0\},$$   $$L_0^2(\Omega)=\{q\in L^2(\Omega); \int_{\Omega} qdx=0\}.$$

The Brinkman equations \eqref{model} model fluid flow in porous media with fractures and generalize the Stokes equations, which approximate Navier–Stokes flow at low Reynolds numbers. Accurate simulation in such multiphysics environments is vital for applications like industrial filtration, foam design, and flow in fractured reservoirs.
High-contrast permeability leads to sharp spatial variations in flow, with regions governed by either Darcy or Stokes behavior. This transition poses a numerical challenge: standard finite element methods often lose stability or accuracy outside their target regime. For example, Stokes-stable elements (e.g., $P_2–P_0$, Crouzeix–Raviart, Mini) degrade in Darcy zones, while Darcy-stable elements (e.g., Raviart–Thomas) perform poorly in Stokes-dominated areas \cite{12}.
A key difficulty is constructing discretizations that remain stable and accurate across both regimes. Recent efforts address this by modifying classical elements to ensure uniform stability for the full Brinkman model \cite{b1,11,12}.

The  WG  finite element method provides a flexible framework for approximating partial differential equations by interpreting differential operators in a weak sense, inspired by distribution theory. Unlike classical methods, WG allows for discontinuous, piecewise polynomial functions and reduces regularity requirements through weak derivatives and stabilization techniques.
WG methods have seen extensive development and successful application to a wide range of PDEs over the past decade \cite{wg1, wg2, wg3, wg4, wg5, wg6, wg7, wg8, wg9, wg10, wg11, wg12, wg13, wg14, fedi, wg15, wg16, wg17, wg18, wg19, wg20, wg21, itera, wy3655}, thanks to their use of weak continuity and natural alignment with variational structures. This flexibility enables robust performance across diverse problems, including those with complex geometries and mixed regimes.
This paper introduces a WG formulation designed to accommodate general polygonal and polyhedral meshes. A thorough theoretical analysis is carried out to  establish optimal-order error estimates, demonstrating the method’s precision and reliability.

The structure of the paper is as follows. Section 2 offers a brief overview of the weak gradient and weak divergence operators, along with their corresponding discrete versions. In Section 3, we develop a  WG  method tailored for the Brinkman equations, accommodating general polytopal meshes. Section 4 is dedicated to demonstrating the well-posedness of the proposed numerical scheme by proving the existence and uniqueness of the solution. As a key analytical component, an inf-sup condition is also established. In Section 5, we derive the error equations related to the WG formulation. Section 6 presents an optimal error bound for the numerical solution, while Section 7 extends the analysis to derive convergence rates in the $L^2$-norm. Section 8 concludes with numerical experiments that validate the theoretical predictions and illustrate the method's accuracy and robustness.

Standard mathematical notation is employed throughout this work. Let $D\subset \mathbb{R}^d$  be an open and bounded domain with a Lipschitz-continuous boundary. For any integer $s\geq 0$, we denote the inner product, seminorm, and norm in the Sobolev space $H^s(D)$ by  $(\cdot,\cdot)_{s,D}$, $|\cdot|_{s,D}$ and $\|\cdot\|_{s,D}$ respectively. When $D=\Omega$, we omit the subscript for simplicity. Additionally, when $s=0$, these notations reduce to  $(\cdot,\cdot)_D$, $|\cdot|_D$ and $\|\cdot\|_D$, respectively.

\section{Discrete Weak Differential Operators} 
This section revisits the definitions of the weak gradient and weak divergence operators, along with their corresponding discrete forms, as developed in \cite{autostokes}.

Let $T$ be a polytopal element with boundary  $\partial T$. A weak function on $T$  is defined as a pair  $\bv=\{\bv_0, \bv_b\}$, where $\bv_0\in [L^2(T)]^d$ represents the interior component, and $\bv_b\in [L^{2}(\partial T)]^d$ denotes the boundary component. Notably, the boundary part 
  $\bv_b$ is not required to match the trace of  $\bv_0$ on $\partial T$.  
 
The space of all weak functions on   $T$, denote by  $W(T)$, is given by
 \begin{equation*}\label{2.1}
 W(T)=\{\bv=\{\bv_0, \bv_b\}: \bv_0\in [L^2(T)]^d, \bv_b\in [L^{2}(\partial
 T)]^d\}.
\end{equation*}
 
 The weak gradient $\nabla_w\bv$ is a linear operator that maps
 $W(T)$ into the dual space of $[H^1(T)]^{d\times d}$. For any 
 $\bv\in W(T)$, it is defined via:
 \begin{equation*} 
  (\nabla_w\bv, \bvarphi)_T=-(\bv_0, \nabla \cdot \bvarphi)_T+
  \langle \bv_b,   \bvarphi \cdot \bn \rangle_{\partial T},\quad \forall \bvarphi\in [H^1(T)]^{d\times d},
  \end{equation*}
 where $\bn$  denotes the outward unit normal vector to $\partial T$, with components $n_i$ for $i=1,\cdots, d$.

Analogously, the weak divergence $\nabla_w\cdot \bv$ is a linear
 operator mapping 
 $W(T)$ into the dual of $H^1(T)$, defined as
 \begin{equation*} 
  (\nabla_w\cdot \bv, w)_T=-(\bv_0, \nabla  w)_T+
  \langle \bv_b\cdot\bn,  w \rangle_{\partial T},\quad \forall w\in H^1(T).
  \end{equation*} 
 
 For any non-negative integer $r\ge 0$, let $P_r(T)$ denote the space of polynomials of total degree no more than 
 $r$.   The discrete weak gradient 
 $\nabla_{w, r, T}\bv$ is the unique polynomial in  $[P_r(T)]^{d\times d}$ that satisfies:
 \begin{equation}\label{2.4}
(\nabla_{w, r, T}\bv, \bvarphi)_T=-(\bv_0, \nabla \cdot \bvarphi)_T+
  \langle \bv_b,   \bvarphi \cdot \bn \rangle_{\partial T},\quad \forall \bvarphi \in [P_r(T)]^{d\times d}.
  \end{equation}
If  $\bv_0\in
 [H^1(T)]^d$ is sufficiently smooth, an integration by parts transforms the above into an equivalent form:
 \begin{equation}\label{2.4new}
(\nabla_{w, r, T}\bv, \bvarphi)_T= (\nabla \bv_0,  \bvarphi)_T+
  \langle \bv_b-\bv_0,   \bvarphi \cdot \bn \rangle_{\partial T}, \quad \forall \bvarphi \in [P_r(T)]^{d\times d}.
  \end{equation}  

Similarly, the discrete weak divergence $\nabla_{w, r, T}\cdot\bv$ is defined as the unique polynomial in $P_r(T)$ such that
 \begin{equation}\label{div}
(\nabla_{w, r, T}\cdot \bv, w)_T=-(\bv_0, \nabla  w)_T+
  \langle \bv_b\cdot\bn,  w \rangle_{\partial T},\quad \forall w\in P_r(T).
  \end{equation}
Again, for smooth $\bv_0\in
 [H^1(T)]^d$, this is equivalent to:
 \begin{equation}\label{divnew}
(\nabla_{w, r, T}\cdot \bv, w)_T= (\nabla   \cdot\bv_0,  w)_T+
  \langle (\bv_b-\bv_0)\cdot\bn,  w \rangle_{\partial T},\quad \forall w\in P_r(T).
  \end{equation}  

\section{Weak Galerkin Finite Element Methods}\label{Section:WGFEM}
 Let ${\cal T}_h$ be a shape-regular partition of the domain 
 $\Omega\subset \mathbb R^d$ into general polytopal elements, as defined in \cite{wy3655}. Denote by  ${\mathcal E}_h$ the collection of all edges (in 2D) or faces (in 3D) of 
 ${\cal T}_h$, and let  ${\mathcal E}_h^0={\mathcal E}_h \setminus
 \partial\Omega$ denote the set of all interior edges or faces. For each element $T\in {\cal T}_h$, let $h_T$ be the diameter of $T$, and define the mesh size by
 $$h=\max_{T\in {\cal
 T}_h}h_T.$$

Fix an integer  $k\geq 1$. On each element  $T\in\T_h$, we define the local weak finite element space by
 \begin{equation*}
 V(k, T)=\{\{\bv_0,\bv_b\}: \bv_0\in [P_k(T)]^d, \bv_b\in [P_{k}(e)]^d, e\subset \partial T\}.    
 \end{equation*}
The global weak finite element space $V_h$ is obtained by assembling the local spaces $V(k, T)$ over all elements  $T\in {\cal T}_h$, with the requirement that the boundary part $\bv_b$ is single-valued across shared edges or faces:
\an{\label{V-h}
 V_h=\big\{\{\bv_0,\bv_b\}:\ \{\bv_0,\bv_b\}|_T\in V(k, T),
 \forall T\in {\cal T}_h \big\}. }
The subspace of $V_h$ consisting of functions with vanishing boundary values on $\partial\Omega$ is denoted by
$$
V_h^0=\{\bv\in V_h: \bv_b|_{\partial\Omega}=0\}.
$$
The finite element space for the pressure variable is defined as
\an{\label{W-h} W_h=\{q\in L_0^2(\Omega): q|_T \in P_{k-1}(T)\}. }

For simplicity in notation, we define the discrete weak gradient  $\nabla_{w} \bv$  and discrete weak divergence $\nabla_{w} \cdot\bv$ elementwise as follows, in accordance with equations \eqref{2.4} and \eqref{div}: 
\an{\label{wg-k}
(\nabla_{w} \bv)|_T &= \nabla_{w, k-1, T}(\bv |_T),   \\
\label{wd-k}
(\nabla_{w}\cdot \bv)|_T &= \nabla_{w, k-1, T}\cdot (\bv |_T). }
 
On each element $T\in\T_h$, let $Q_0$ denote the $L^2$ projection onto $P_k(T)$, and for each edge or face $e\subset\partial T$, let $Q_b$ denote the $L^2$ projection onto $P_{k}(e)$. Then, for any $\bv\in [H^1(\Omega)]^d$,  the $L^2$ projection  into  $V_h$ is defined locally by
 $$
  (Q_h \bv)|_T:=\{Q_0(\bv|_T),Q_b(\bv|_{\pT})\},\qquad \forall T\in\T_h.
$$

We now formulate a Weak Galerkin method for approximating solutions to the Brinkman equations \eqref{model}.
\begin{algorithm}\label{PDWG1}
Find $\bu_h=\{\bu_0, \bu_b\} \in V_h^0$ and $p_h\in W_h$  such that  
\begin{equation}\label{WG}
\begin{split} a_h(\bu_h,  \bv_h) -(\nabla_w \cdot\bv_h, p_h)=&(\bf, \bv_0) \qquad  \forall \bv_h\in V_h^0,\\
     (\nabla_w \cdot\bu_h, q_h)=&0 \qquad\qquad \forall q_h\in W_h, 
\end{split}
\end{equation}
where  
\a{ a_h(\bu_h,  \bv_h) & =
  (\nabla_w \bu_h, \nabla_w \bv_h) +(\kappa^{-1} \bu_0,\bv_0)+s(\bu_h,\bv_h), \\
  s(\bu_h, \bv_h) &=\sum_{T\in {\cal T}_h} h_T^{-1} \langle \bu_0-\bu_b, \bv_0-\bv_b \rangle_{\partial T}, }
and 
$$
(\cdot, \cdot)=\sum_{T\in {\cal T}_h}(\cdot, \cdot)_T.
$$
\end{algorithm}

\section{Existence and Uniqueness of the Solution} 
Let ${\cal T}_h$  be a shape-regular finite element mesh partitioning the domain $\Omega$. For any element $T\in {\cal T}_h$ and any function $\phi\in H^1(T)$, the following trace inequality holds  (see \cite{wy3655}):
\begin{equation}\label{tracein}
 \|\phi\|^2_{\partial T} \leq C(h_T^{-1}\|\phi\|_T^2+h_T \|\nabla \phi\|_T^2).
\end{equation}
In particular, if $\phi$ is a polynomial function on $T$, a simplified trace inequality applies (see \cite{wy3655}):
\begin{equation}\label{trace}
\|\phi\|^2_{\partial T} \leq Ch_T^{-1}\|\phi\|_T^2.
\end{equation}

\begin{lemma}\cite{wg21, autostokes}\label{lem}
Let ${\cal T}_h$ be a shape-regular finite element partition of  $\Omega$. Then for any integers $0\leq s \leq 1$, $0\leq m \leq k$,  and $0\leq  n \leq   k-1$, the following approximation estimates hold:
\begin{eqnarray}
\label{error1}
 \sum_{T\in {\cal T}_h} h_T^{2s}\|({\cal Q}_h-I)p\|^2_{s,T}&\leq& C  h^{2n+2}\|p\|^2_{n+1},\\
\label{error2}
\sum_{T\in {\cal T}_h}h_T^{2s}\|\bu- Q _0\bu\|^2_{s,T}&\leq& C h^{2(m+1)}\|\bu\|^2_{m+1}, 
 \\ \label{error3}
\sum_{T\in {\cal T}_h}h_T^{2s}\|\nabla\bu-{\cal Q}_h(\nabla\bu)\|^2_{s,T}&\leq& C h^{2n+2}\|\bu\|^2_{n+2}.
\end{eqnarray}
 \end{lemma}
 
 We define the following norm on the space $V_h$ for any function $\bv = {\bv_0, \bv_b} \in V_h$:
\begin{equation}\label{3norm}
\3bar \bv\3bar= \Big( \sum_{T\in {\cal T}_h}(\nabla_{w} \bv, \nabla_{ w} \bv)_T+(\kappa^{-1} \bv_0,\bv_0)_T +h_T^{-1}\|\bv_0-\bv_b\|_{\partial T}^2\Big) ^{\frac{1}{2}}.
\end{equation}

\begin{lemma}\label{normdefine}
The quantity defined in \eqref{3norm} is a norm on the weak finite element space $V_h$.
\end{lemma}
 \begin{proof}
 We verify the definiteness of the norm. Suppose   $\3bar \cdot\3bar=0$ for some $\bv\in V_h^0$. Then
 $$
  \sum_{T\in {\cal T}_h}(\nabla_{w} \bv, \nabla_{ w} \bv)_T+(\kappa^{-1} \bv_0,\bv_0)_T +h_T^{-1}\|\bv_0-\bv_b\|_{\partial T}^2=0,
$$
which implies $ \nabla_{ w} \bv=0$, $\bv_0=0$ in each $T$, $\bv_0=\bv_b$ on $\partial T$ . It follows that $\bv_0\equiv 0$ in   $\Omega$ and hence $\bv_b\equiv 0$, so $\bv_h\equiv 0$. This completes the proof.
 \end{proof}

Let ${\cal Q}_h$ denote the $L^2$ projection onto the space of piecewise polynomials of degree at most $k-1$ on each element. 

\begin{lemma}\label{Lemma5.1} \cite{autostokes}  For any $\bu \in [H^1(T)]^d$, the following identities hold:
\begin{equation}\label{pro4}
\nabla_w  (Q_h\bu)={\cal Q}_h(\nabla  \bu),
 \end{equation}
 \begin{equation}\label{pro3}
\nabla_w\cdot Q_h\bu={\cal Q}_h(\nabla \cdot \bu).
\end{equation}
\end{lemma}
 
 We now establish the inf-sup condition for the bilinear form $b(\cdot, \cdot)$.
 \begin{lemma}  
   (Inf-Sup Condition)  There exists a constant $C>0$,   independent of the mesh size $h$, such that for all $\zeta\in W_h$, 
     \begin{equation}\label{infsup}
         \sup_{\bv\in V_h^0} \frac{(\nabla_w \cdot \bv, \zeta)}{\3bar \bv\3bar}\geq C\|\zeta\|.
     \end{equation}
 \end{lemma}
\begin{proof}
For any  $\zeta\in W_h\subset L_0^2(\Omega)$, it is known (see \cite{2, 4, 5, 6, 7}) that there exists  $\bar{\bv}\in [H_0^1(\Omega)]^d$ satisfying:
\begin{equation} \label{eqq}
\frac{(\nabla\cdot\bar{\bv}, \zeta)}{\|\bar{\bv}\|_1}\geq C\|\zeta\|.
\end{equation}
Define $\bv=Q_h\bar{\bv}\in V_h^0$.  We show that
\begin{equation}\label{err1}
    \3bar \bv\3bar \leq C\|\bar{\bv}\|_1.
\end{equation}
Using identity \eqref{pro4}, we estimate:
$$
\sum_{T\in {\cal T}_h}\|\nabla_w \bv\|_T^2=\sum_{T\in {\cal T}_h}\|\nabla_w Q_h\bar{\bv}\|_T^2=\sum_{T\in {\cal T}_h}\|\cal{Q}_h\nabla  \bar{\bv}\|_T^2\leq \sum_{T\in {\cal T}_h}\| \nabla  \bar{\bv}\|_T^2.
$$
Also,  
\begin{equation*}
    \begin{split}
    \sum_{T\in {\cal T}_h} (\kappa^ {-1}\bv_0, \bv_0)_T =\sum_{T\in {\cal T}_h}(\kappa ^ {-1}Q_0\bar{\bv}, Q_0\bar{\bv})_T \leq     \sum_{T\in {\cal T}_h} \|\bar{\bv}\|_T^2.  
    \end{split}
\end{equation*}
Applying trace inequality \eqref{tracein} and estimate \eqref{error2} with $m=0$ and $s=0, 1$, we obtain:
 \begin{equation*}
    \begin{split}
     \sum_{T\in {\cal T}_h} h_T^{-1}\|\bv_0-\bv_b\|^2_{\partial T}   =& \sum_{T\in {\cal T}_h} h_T^{-1}\|Q_0\bar{\bv}-Q_b\bar{\bv}\|^2_{\partial T}\\
     \leq& \sum_{T\in {\cal T}_h} h_T^{-1}\|Q_0\bar{\bv}- \bar{\bv}\|^2_{\partial T}\\
     \leq & \sum_{T\in {\cal T}_h} h_T^{-2}\|Q_0\bar{\bv}- \bar{\bv}\|^2_{T}+\|Q_0\bar{\bv}- \bar{\bv}\|^2_{1,T}\\
     \leq & C\|\bar{\bv} \|_1^2.
     \end{split}
\end{equation*}   
Combining these gives \eqref{err1}.

Using \eqref{pro3}, we find:
$$
(\nabla_w \cdot  \bv, \zeta)_T=(\nabla_w \cdot Q_h\bar{\bv}, \zeta)_T=( \cal{Q}_h \nabla \cdot\bar{\bv}, \zeta)_T=(  \nabla\cdot\bar{\bv}, \zeta)_T.
$$
Combining the estimate above with \eqref{eqq} and \eqref{err1}, we obtain
$$
\frac{(\nabla_w \cdot \bv, \zeta)_T}{\3bar \bv\3bar} \geq C \frac{(\nabla  \cdot \bar{\bv}, \zeta)_T}{ \| \bar{\bv}\|_1 } \geq C\|\zeta\|.
$$

This completes the proof of the Lemma.

\end{proof}
 
\begin{theorem}The weak Galerkin Algorithm \ref{PDWG1} for the Brinkman equations \eqref{model} admits a unique solution.
\end{theorem}
\begin{proof}
Assume that there exist two distinct solutions  $(\bu_h^{(1)}, p_h^{(1)})$ and $(\bu_h^{(2)}, p_h^{(2)})$ in $V_h^0\times W_h$.  Define their difference: $${\Xi}_{\bu_h}= \{{\Xi}_{\bu_0}, {\Xi}_{\bu_b}\}=\bu_h^{(1)}-\bu_h^{(2)}\in V_h^0, \qquad {\Xi}_{p_h}=p_h^{(1)}-p_h^{(2)}\in W_h.$$  Then $({\Xi}_{\bu_h}, {\Xi}_{p_h})$ satisfies the homogeneous system:
\begin{equation}\label{wgstokes}
\begin{split}
  (\nabla_w {\Xi}_{\bu_h}, \nabla_w \bv_h)+(\kappa^{-1} {\Xi}_{\bu_0}, \bv_0)+s({\Xi}_{\bu_h}, \bv_h)-(\nabla_w \cdot\bv_h, {\Xi}_{p_h}) =&0, \qquad \qquad \forall \bv_h\in V_h^0,\\
     (\nabla_w \cdot {\Xi}_{\bu_h}, q_h)=&0, \qquad\qquad \forall q_h\in W_h.
\end{split}
\end{equation}
Taking test functions $\bv_h={\Xi}_{\bu_h}$ and $q_h={\Xi}_{p_h}$ in \eqref{wgstokes} gives: 
$$\3bar {\Xi}_{\bu_h} \3bar=0.$$ 

By Lemma \ref{normdefine}, we deduce ${\Xi}_{\bu_h}\equiv 0$.
  Substituting ${\Xi}_{\bu_h}\equiv 0$ into the first equation of \eqref{wgstokes} yields:
$$
(\nabla_w\cdot \bv_h, {\Xi}_{p_h})=0, \qquad  \forall \bv_h\in V_h^0.
$$
By the inf-sup condition \eqref{infsup}, this implies
$\|{\Xi}_{p_h}\|=0$, i.e., $ {\Xi}_{p_h} \equiv  0$.

Thus, we conclude that
 $\bu_h^{(1)}\equiv \bu_h^{(2)}$ and $p_h^{(1)}\equiv p_h^{(2)}$, proving uniqueness.  
\end{proof}

\section{Error Equations}
Let  $\bu$ and $p$  represent the exact solutions to the Brinkman problem \eqref{model}, and let 
  $\bu_h \in V_{h}^0$ and $p_h\in W_h$ denote the corresponding discrete solutions produced by the WG Algorithm \ref{PDWG1}. We define the error functions  $e_{\bu_h}$ and $e_{p_h}$ for velocity and pressure as:
\begin{equation}\label{error} 
e_{\bu_h}=Q_h\bu-\bu_h,  \qquad e_{p_h}={\cal Q}_hp-p_h.
\end{equation}

\begin{lemma}\label{errorequa}
The error functions $e_{\bu_h}$ and $e_{p_h}$ defined in \eqref{error} satisfy the following system of error equations:
\begin{equation}\label{erroreqn}
\begin{split}
 &(\nabla_{w} e_{\bu_h}, \nabla_w  \bv_h) + s (e_{\bu_h}, \bv_h)+(\kappa^{-1} e_{\bu_0}, \bv_0) -(\nabla_w \cdot \bv_h, e_{p_h})  \\
 =&\ell_1 (\bu, \bv_h)+\ell_2 (  \bv_h, p)+s(Q_h\bu, \bv_h),  \qquad  \forall \bv_h\in V_h^0, \\& (\nabla_w \cdot e_{\bu_h}, q_h)= 0, \qquad \forall q_h\in W_h,
\end{split}
\end{equation}
where $\ell_1$ and $\ell_2$ are given by:
$$
\ell_1 (\bu, \bv_h)=\sum_{T\in {\cal T}_h}  
  \langle \bv_b-\bv_0,   ({\cal Q}_h-I) \nabla \bu \cdot \bn \rangle_{\partial T},
$$ 
$$
\ell_2 (  \bv_h, p)=\sum_{T\in {\cal T}_h} -\langle ({\cal Q}_h -I)p, (\bv_b-\bv_0)\cdot \bn \rangle_{\partial T}.
$$
\end{lemma}
\begin{proof}  Using identity \eqref{pro4} and integrating by parts, we choose $\bvarphi= {\cal Q}_h \nabla \bu$ in  \eqref{2.4new}, yielding:
\begin{equation}\label{term1}
\begin{split}
&\sum_{T\in {\cal T}_h} (\nabla_{w} Q_h\bu, \nabla_w  \bv_h)_T\\=&\sum_{T\in {\cal T}_h}\sum_{T\in {\cal T}_h} ({\cal Q}_h(\nabla \bu), \nabla_w  \bv_h)_T\\ 
=&\sum_{T\in {\cal T}_h}   (\nabla \bv_0,  {\cal Q}_h \nabla \bu)_T+
  \langle \bv_b-\bv_0,   {\cal Q}_h \nabla \bu \cdot \bn \rangle_{\partial T}\\
=& \sum_{T\in {\cal T}_h}(\nabla \bv_0,   \nabla \bu)_T+
  \langle \bv_b-\bv_0,   {\cal Q}_h \nabla \bu \cdot \bn \rangle_{\partial T}\\
  =& \sum_{T\in {\cal T}_h}-(  \bv_0,   \Delta \bu)_T+\langle \nabla \bu\cdot\bn, \bv_0\rangle_{\partial T}+
  \langle \bv_b-\bv_0,   {\cal Q}_h \nabla \bu \cdot \bn \rangle_{\partial T}\\
  =& \sum_{T\in {\cal T}_h}-(  \bv_0,   \Delta \bu)_T+ 
  \langle \bv_b-\bv_0,   ({\cal Q}_h-I) \nabla \bu \cdot \bn \rangle_{\partial T},
\end{split}
\end{equation}
where the boundary contribution  $\sum_{T\in {\cal T}_h} \langle \nabla \bu\cdot\bn, \bv_b\rangle_{\partial T}=\langle \nabla \bu\cdot\bn, \bv_b\rangle_{\partial \Omega}=0$ because $\bv_b=0$ on $\partial \Omega$.

Applying   \eqref{divnew} with  $w={\cal Q}_h  p$ and using integration by parts, we get:
 \begin{equation}\label{termm}
     \begin{split}
  &\sum_{T\in {\cal T}_h} (\nabla_w \cdot \bv_h, {\cal Q}_h p)_T \\ =& \sum_{T\in {\cal T}_h} (\nabla \cdot \bv_0,  {\cal Q}_h p)_T + \langle {\cal Q}_h p, (\bv_b-\bv_0)\cdot \bn \rangle_{\partial T}\\
  =& \sum_{T\in {\cal T}_h} (\nabla \cdot \bv_0,    p)_T + \langle {\cal Q}_h p, (\bv_b-\bv_0)\cdot \bn \rangle_{\partial T}\\
  =& \sum_{T\in {\cal T}_h}-( \bv_0,    \nabla p)_T 
+\langle p, \bv_0\cdot\bn\rangle_{\partial T}+ \langle {\cal Q}_h p, (\bv_b-\bv_0)\cdot \bn \rangle_{\partial T}\\    =& \sum_{T\in {\cal T}_h}-( \bv_0, \nabla p)_T + \langle ( {\cal Q}_h -I)p, (\bv_b-\bv_0)\cdot \bn \rangle_{\partial T},
\end{split}
 \end{equation}
where we again used the vanishing trace of $\bv_b$ on $\partial \Omega$.

Subtracting \eqref{termm} from \eqref{term1} and incorporating the first equation from \eqref{model}, we derive:
\begin{equation*}  
\begin{split}
&\sum_{T\in {\cal T}_h}(\nabla_{w} Q_h\bu, \nabla_w  \bv_h)_T  +(\kappa ^{-1} Q_0\bu, \bv_0)_T+s(Q_h\bu, \bv_h)-(\nabla_w \cdot \bv_h, {\cal Q}_hp)_T \\=& \sum_{T\in {\cal T}_h}-(  \bv_0,   \Delta \bu)_T+ 
  \langle \bv_b-\bv_0,   ({\cal Q}_h-I) \nabla \bu \cdot \bn \rangle_{\partial T} + 
 ( \bv_0, \nabla p)_T \\&-\langle ({\cal Q}_h -I)p, (\bv_b-\bv_0)\cdot \bn \rangle_{\partial T}+(\kappa ^{-1}  \bu, \bv_0)_T +s(Q_h\bu, \bv_h)\\
 =&\sum_{T\in {\cal T}_h} (  \bv_0,   \bf)_T+ 
  \langle \bv_b-\bv_0,   ({\cal Q}_h-I) \nabla \bu \cdot \bn \rangle_{\partial T}   \\&-\langle ({\cal Q}_h -I)p, (\bv_b-\bv_0)\cdot \bn \rangle_{\partial T}+s(Q_h\bu, \bv_h).
  \end{split}
\end{equation*}
Subtracting the first equation of  \eqref{WG}  from the above   leads to the first error equation in \eqref{erroreqn}.

Note that using \eqref{pro3} and  the second equation of \eqref{model}, we have:
$$
0= (\nabla_w \cdot Q_h\bu, q_h)=\sum_{T\in {\cal T}_h}({\cal Q}_h \nabla\cdot \bu, q_h)_T=\sum_{T\in {\cal T}_h}(\nabla\cdot \bu, q_h)_T=0.
$$
Subtracting this from the second equation of \eqref{WG} yields the second part of \eqref{erroreqn}, completing the proof.

\end{proof}

\section{Error Estimates} 
\begin{lemma}
Let  $\bu\in [H^{k+1}(\Omega)]^d$, $q\in H^{k}(\Omega)$, and let $\bv_h=\{\bv_0, \bv_b\}\in V_h^0$,   $q_h \in W_h$. Then the following bounds hold:
\begin{equation}\label{es1}
   |\ell_1(\bu, \bv_h)| \leq  Ch^k \|\bu\|_{k+1}\3bar \bv_h \3bar, 
\end{equation}
\begin{equation}\label{es2}
   |\ell_2(\bv_h, p)| \leq  Ch^k \|p\|_{k}\3bar \bv_h \3bar,
\end{equation}
\begin{equation}\label{es3}
   |s(Q_h\bu, \bv_h)| \leq  Ch^k \|\bu\|_{k+1}\3bar \bv_h \3bar.
\end{equation}
\end{lemma}

\begin{proof}
To derive \eqref{es1}, we apply the Cauchy–Schwarz inequality, the trace inequality \eqref{tracein}, and the approximation estimate \eqref{error3} with $n = k-1$:
 \begin{equation*} 
\begin{split}
|\ell_1(\bu, \bv_h)|\leq &(\sum_{T\in {\cal T}_h}  
  h_T^{-1}\|\bv_b-\bv_0\|_{\partial T} ^2)^{\frac{1}{2}} (\sum_{T\in {\cal T}_h}  
 h_T \|({\cal Q}_h-I) \nabla \bu \cdot \bn\|_{\partial T} ^2)^{\frac{1}{2}} \\
\leq & \3bar \bv_h \3bar(\sum_{T\in {\cal T}_h}  
 \|({\cal Q}_h-I) \nabla \bu \cdot \bn\|_{T} ^2+ h_T^2 \|({\cal Q}_h-I) \nabla \bu \cdot \bn\|_{1, T} ^2)^{\frac{1}{2}}\\ 
 \leq & Ch^k \|\bu\|_{k+1}\3bar \bv_h \3bar.
\end{split}
\end{equation*} 

To estimate \eqref{es2}, we again use the Cauchy–Schwarz inequality, the trace inequality \eqref{tracein},   along with \eqref{error1} with $n=k-1$:
 \begin{equation*} 
\begin{split}
|\ell_2(\bv_h, p)|\leq &(\sum_{T\in {\cal T}_h}  
  h_T^{-1}\|\bv_b-\bv_0\|_{\partial T} ^2)^{\frac{1}{2}} (\sum_{T\in {\cal T}_h}  
 h_T \|({\cal Q}_h-I)p\|_{\partial T} ^2)^{\frac{1}{2}} \\
\leq & \3bar \bv_h \3bar (\sum_{T\in {\cal T}_h}  
 \|({\cal Q}_h-I)p\|_{T} ^2+ h_T^2 \|({\cal Q}_h-I)p\|_{1, T} ^2)^{\frac{1}{2}}\\ 
 \leq & Ch^k \|p\|_{k}\3bar \bv_h \3bar.
\end{split}
\end{equation*}

For \eqref{es3}, using the Cauchy–Schwarz inequality, the trace inequality \eqref{tracein},  and estimate \eqref{error2} with $m=k$ yields:
\begin{equation*} 
\begin{split}
    |s(Q_h\bu, \bv_h)| =&|\sum_{T\in {\cal T}_h} h_T^{-1} \langle Q_0\bu-Q_b\bu, \bv_0-\bv_b\rangle_{\partial T}| \\ \leq & \Big( \sum_{T\in {\cal T}_h} h_T^{-1} \| Q_0\bu-Q_b\bu\|^2_{\partial T}\Big)^{\frac{1}{2}}  \Big(  \sum_{T\in {\cal T}_h} h_T^{-1} \|\bv_0-\bv_b\|^2_{\partial T}\Big)^{\frac{1}{2}}
    \\ \leq & \Big(  \sum_{T\in {\cal T}_h} h_T^{-1} \| Q_0\bu- \bu\|^2_{\partial T}\Big)^{\frac{1}{2}}  \3bar \bv_h\3bar
  \\ \leq & \Big(  \sum_{T\in {\cal T}_h} h_T^{-2} \| Q_0\bu- \bu\|^2_{T}+ \| Q_0\bu- \bu\|^2_{1, T}\Big)^{\frac{1}{2}}  \3bar \bv_h\3bar \\ 
 \leq & Ch^k \|\bu\|_{k+1}\3bar \bv_h \3bar.
\end{split}   
\end{equation*}

This completes the proof of the lemma.
\end{proof}
\begin{theorem}
Let  
$(\bu, p)$ be the exact solution of the Brinkman system \eqref{model} with regularity   
$\bu\in [H^{k+1}(\Omega)]^d$ and $p\in H^k(\Omega)$. Let $(\bu_h, p_h)$  be the corresponding numerical solution produced by the  WG method \ref{PDWG1}. Then, the following error estimate holds:
\begin{equation}\label{trinorm}
\3bar e_{\bu_h}\3bar +\|e_{p_h}\|\leq Ch^{k}(\|\bu\|_{k+1}+\|p\|_k).
\end{equation}
\end{theorem}
\begin{proof}
Choose test functions $\bv_h=e_{\bu_h}$ and $q_h=e_{p_h}$ in the error equations \eqref{erroreqn}. Using the bounds from \eqref{es1}-\eqref{es3}, we obtain:
 \begin{equation*} 
\begin{split}
\3bar e_{\bu_h} \3bar^2 =&\ell_1(\bu, e_{\bu_h})+\ell_2(e_{\bu_h}, p)+s(Q_h\bu, e_{\bu_h})\\
\leq & Ch^{k}(\|\bu\|_{k+1}+\|p\|_k) \3bar e_{\bu_h} \3bar,
\end{split}
\end{equation*}
which implies
\begin{equation}\label{esteh}
    \3bar e_{\bu_h} \3bar\leq   Ch^{k}(\|\bu\|_{k+1}+\|p\|_k).
\end{equation}

To estimate $\|e_{p_h}\|$, we revisit the first equation in \eqref{erroreqn}, apply the Cauchy–Schwarz inequality, the estimates \eqref{es1}-\eqref{es3}, and \eqref{esteh} gives
\begin{equation*} 
\begin{split}
    &|(\nabla_w \cdot \bv_h, e_{p_h})|  \\
 =&|-\ell_1 (\bu, \bv_h)-\ell_2 (  \bv_h, p)-s(Q_h\bu, \bv_h)  +
 (\nabla_{w} e_{\bu_h}, \nabla_w  \bv_h) + s (e_{\bu_h}, \bv_h)+(\kappa^{-1} e_{\bu_0}, \bv_0)|
 \\ \leq & Ch^{k}(\|\bu\|_{k+1}+\|p\|_k) \3bar \bv_h \3bar +\3bar e_{\bu_h}\3bar  \3bar \bv_h \3bar \\
 \leq & Ch^{k}(\|\bu\|_{k+1}+\|p\|_k) \3bar \bv_h \3bar.\end{split}
\end{equation*}
By the inf-sup condition \eqref{infsup}, this yields
\begin{equation}\label{estph}
    \|e_{p_h} \| \leq Ch^{k}(\|\bu\|_{k+1}+\|p\|_k).
\end{equation}

Combining \eqref{esteh} and \eqref{estph} proves the theorem.
\end{proof}

\section{Error Analysis in the $L^2$ Norm}
To derive an error estimate in the $L^2$ norm, we employ a classical duality argument. Recall that the velocity error is given by:
$$e_{\bu_h}=\bu-\bu_h=\{e_{\bu_0}, e_{\bu_b}\}=\{Q_0\bu-\bu_0, Q_b\bu-\bu_b\}.$$
We now consider the corresponding dual problem associated with the Brinkman system \eqref{model}. Specifically, we seek a pair  $(\bw, q) \in [H^2(\Omega)]^d \times H^1(\Omega)$ satisfying:
\begin{equation}\label{dual}
\begin{split}
   - \Delta \bw+\kappa^{-1} \bw+\nabla q &=e_{\bu_0}, \qquad \text{in}\ \Omega,\\
\nabla \cdot\bw&=0,\qquad\quad \text{in}\ \Omega,  \\
\bw&=0, \quad \qquad\text{on}\ \partial\Omega.
    \end{split}
\end{equation}
We assume the dual solution satisfies the regularity bound:
\begin{equation}\label{regu2}
 \|\bw\|_2+\|q\|_1\leq C\|e_{\bu_0}\|.
 \end{equation}
 
 \begin{theorem}
Let $(\bu, p)\in [H^{k+1}(\Omega)]^d \times H^k(\Omega)$ be the exact solutions to the Brinkman problem \eqref{model}, and let
 $(\bu_h, p_h)\in V_h^0 \times  W_h$ be its approximation via the Weak Galerkin Algorithm \ref{PDWG1}. Under the regularity condition \eqref{regu2}, there exists a constant $C$ such that:
\begin{equation*}
\|Q_0\bu-\bu_0\|\leq Ch^{k+1}(\|\bu\|_{k+1}+\|p\|_k).
\end{equation*}
 \end{theorem}
 
 \begin{proof}
We test the first equation of the dual problem \eqref{dual} with  $e_{\bu_0}$:
 \begin{equation}\label{e1}
 \begin{split}
 \|e_{\bu_0}\|^2 =&(- \Delta \bw+\kappa^{-1}\bw+\nabla q, e_{\bu_0}).
 \end{split}
 \end{equation}
Using identity \eqref{term1} with  $\bu=\bw$ and  $\bv_h=e_{\bu_h}$, we get:
$$
\sum_{T\in {\cal T}_h}(- \Delta \bw, e_{\bu_0})_T=\sum_{T\in {\cal T}_h} (\nabla_w Q_h\bw, \nabla_w e_{\bu_h})_T-\langle e_{\bu_b}-e_{\bu_0}, ({\cal Q}_h-I)\nabla \bw\cdot\bn \rangle_{\partial T}.
$$
Similarly, from identity \eqref{termm} with  $p=q$ and $\bv_h=e_{\bu_h}$:
$$
\sum_{T\in {\cal T}_h} (\nabla q, e_{\bu_0})=\sum_{T\in {\cal T}_h} -(\nabla_w \cdot e_{\bu_h}, {\cal Q}_hq)_T+\langle ({\cal Q}_h-I)q, (e_{\bu_b}-e_{\bu_0})\cdot\bn\rangle_{\partial T}.
$$
Substituting the above two identities into the identity  \eqref{e1} and using  the error equation \eqref{erroreqn} with $\bv_h=Q_h\bw$ and $q_h={\cal Q}_hq$, we obtain:
\begin{equation}\label{e2}
 \begin{split}
 &\|e_{\bu_0}\|^2 \\=&\sum_{T\in {\cal T}_h} (\nabla_wQ_h\bw, \nabla_w e_{\bu_h})_T-\langle e_{\bu_b}-e_{\bu_0}, ({\cal Q}_h-I)\nabla \bw\cdot\bn \rangle_{\partial T}\\&-(\nabla_w \cdot e_{\bu_h}, {\cal Q}_hq)_T+\langle ({\cal Q}_h-I)q, (e_{\bu_b}-e_{\bu_0})\cdot\bn\rangle_{\partial T}+(\kappa^{-1}Q_0 \bw, e_{\bu_0})_T\\
 &+s(Q_h\bw, e_{\bu_h})-s(Q_h\bw, e_{\bu_h})\\
 =& \ell_1(\bu, Q_h\bw)+\ell_2(Q_h\bw, p)+s(Q_h\bu, Q_h\bw)+\sum_{T\in {\cal T}_h} (\nabla_w \cdot(Q_h\bw), e_{p_h})_T\\
 &-\langle e_{\bu_b}-e_{\bu_0}, ({\cal Q}_h-I)\nabla \bw\cdot\bn \rangle_{\partial T}\\&+\langle ({\cal Q}_h-I)q, (e_{\bu_b}-e_{\bu_0})\cdot\bn\rangle_{\partial T}-s(Q_h\bw, e_{\bu_h})\\
 =& \ell_1(\bu, Q_h\bw)+\ell_2(Q_h\bw, p) +\sum_{T\in {\cal T}_h} (\nabla_w \cdot(Q_h\bw), e_{p_h})_T\\
 &-\langle e_{\bu_b}-e_{\bu_0}, ({\cal Q}_h-I)\nabla \bw\cdot\bn \rangle_{\partial T}+\langle ({\cal Q}_h-I)q, (e_{\bu_b}-e_{\bu_0})\cdot\bn\rangle_{\partial T}\\&+s(Q_h\bw, \bu_h)
\\
=&\sum_{i=1}^6 I_i.
 \end{split}
 \end{equation}

 We now estimate each term $I_i (i=1, \cdots, 6)$ individually.
 
Estimate of $I_1$:  Using the Cauchy-Schwarz inequality, the trace inequality \eqref{tracein}, along with the estimate \eqref{error2} ($m = 1$) and the estimate  \eqref{error3} ($n = k-1$):
\begin{equation*} 
 \begin{split}
&|\ell_1(\bu, Q_h\bw)|\\=&|\sum_{T\in {\cal T}_h}  
  \langle Q_b\bw-Q_0\bw,   ({\cal Q}_h-I) \nabla \bu \cdot \bn \rangle_{\partial T}|\\
  \leq & (\sum_{T\in {\cal T}_h}  
  \| Q_b\bw-Q_0\bw\|^2_{\partial T})^
{\frac{1}{2}} (\sum_{T\in {\cal T}_h}\| ({\cal Q}_h-I) \nabla \bu \cdot \bn \|^2_{\partial T})
\\
\leq & (\sum_{T\in {\cal T}_h}  
  h_T^{-1}\|  \bw-Q_0\bw\|^2_{T}+h_T\|  \bw-Q_0\bw\|^2_{1, T})^
{\frac{1}{2}}\\
&\cdot(\sum_{T\in {\cal T}_h}  h_T^{-1}\| ({\cal Q}_h-I) \nabla \bu \cdot \bn \|^2_{T}+h_T \| ({\cal Q}_h-I) \nabla \bu \cdot \bn \|^2_{1, T})\\
\leq & Ch^{-1}h^2\|\bw\|_2 h^k\|\bu\|_{k+1}\leq Ch^{k+1}\|\bw\|_2  \|\bu\|_{k+1}.
  \end{split}
 \end{equation*}
Estimate of  $I_2$: Using the Cauchy-Schwarz inequality, the trace inequality \eqref{tracein}, using the estimate \eqref{error1} ($n = k-1$) and the estimate \eqref{error2} ($m = 1$):
\begin{equation*} 
 \begin{split}
&|\ell_2(Q_h\bw, p)| \\=&|\sum_{T\in {\cal T}_h} -\langle ({\cal Q}_h -I)p, (Q_b\bw-Q_0\bw)\cdot \bn \rangle_{\partial T}|\\
\leq & (\sum_{T\in {\cal T}_h}\| ({\cal Q}_h -I)p\|_{\partial T}^2)^\frac{1}{2}(\sum_{T\in {\cal T}_h} \|Q_b\bw-Q_0\bw)\cdot \bn \|_{\partial T}^2)^\frac{1}{2}\\
\leq & (\sum_{T\in {\cal T}_h}h_T^{-1}\| ({\cal Q}_h -I)p\|_{T}^2+h_T\| ({\cal Q}_h -I)p\|_{1, T}^2)^\frac{1}{2}\\&\cdot (\sum_{T\in {\cal T}_h} h_T^{-1}\|(\bw-Q_0\bw)\cdot \bn \|_{T}^2+h_T \|(\bw-Q_0\bw)\cdot \bn \|_{1, T}^2)^\frac{1}{2}\\
\leq &Ch^{-1}h^2\|\bw\|_2 h^k\|p\|_{k} \leq Ch^{k+1} \|\bw\|_2  \|p\|_{k}.
 \end{split}
 \end{equation*} 
 Estimate of  $I_3$: Apply  
\eqref{pro3} and the second equality in \eqref{dual}, we have
\begin{equation*} 
 \begin{split}
 |\sum_{T\in {\cal T}_h} (\nabla_w\cdot (Q_h\bw), e_{p_h})_T|= |\sum_{T\in {\cal T}_h} ({\cal Q}_h\nabla \cdot \bw, e_{p_h})_T| 
   =  0.
  \end{split}
 \end{equation*} 
Estimate of $I_4$:  Applying the Cauchy-Schwarz inequality, trace inequality \eqref{tracein},  estimate \eqref{error3} with $n=0$,  the error estimate   \eqref{trinorm},  gives
\begin{equation*}
\begin{split}
&|\sum_{T\in {\cal T}_h}\langle e_{\bu_b}-e_{\bu_0}, ({\cal Q}_h-I)\nabla \bw\cdot\bn \rangle_{\partial T}|\\
\leq & (\sum_{T\in {\cal T}_h}h_T^{-1}\|e_{\bu_b}-e_{\bu_0}\|_{\partial T}^2)^\frac{1}{2}  (\sum_{T\in {\cal T}_h} h_T\|({\cal Q}_h-I)\nabla \bw\cdot\bn \|_{\partial T}^2)^\frac{1}{2}\\
\leq & \3bar e_{\bu_h}\3bar (\sum_{T\in {\cal T}_h} \|({\cal Q}_h-I)\nabla \bw\cdot\bn \|_{T}^2+h_T^2\|({\cal Q}_h-I)\nabla \bw\cdot\bn \|_{1,T}^2)^\frac{1}{2}\\
 \leq & \3bar e_{\bu_h}\3bar  h\|\bw\|_2
 \\
 \leq & Ch^{k+1}(\|\bu\|_{k+1}+\|p\|_k)\|\bw\|_2.
 \end{split}
 \end{equation*} 
Estimate of $I_5$:  Applying the Cauchy-Schwarz inequality, the trace inequality \eqref{tracein},  estimate \eqref{error1} with $n=0$,  and the error bound  \eqref{trinorm}, leads to
\begin{equation*}
\begin{split}
&|\sum_{T\in {\cal T}_h}\langle ({\cal Q}_h-I)q, (e_{\bu_b}-e_{\bu_0})\cdot\bn\rangle_{\partial T}|\\
 \leq&  (\sum_{T\in {\cal T}_h}h_T\|({\cal Q}_h-I)q\|_{\partial T}  ^2)^\frac{1}{2}  (\sum_{T\in {\cal T}_h}h_T^{-1}\|(e_{\bu_b}-e_{\bu_0})\cdot\bn\|_{\partial T}^2)^\frac{1}{2} \\
  \leq&  (\sum_{T\in {\cal T}_h} \|({\cal Q}_h-I)q\|_{T}  ^2+h_T^2\|({\cal Q}_h-I)q\|_{1, T}  ^2)^\frac{1}{2}  \3bar e_{\bu_h}\3bar\\
    \leq &  Ch\|q\|_1 h^k(\|\bu\|_{k+1}+\|p\|_k). \\
   \end{split}
 \end{equation*} 
Estimate of  $I_6$:  Using the Cauchy-Schwarz inequality, the trace inequality \eqref{tracein}, the estimate \eqref{error2} with $m=1$, and the error bound  \eqref{trinorm}, we get
\begin{equation*}
\begin{split}
|s(Q_h\bw, \bu_h)|=& |\sum_{T\in {\cal T}_h} h_T^{-1}\langle  Q_0\bw-Q_b\bw, \bu_0-\bu_b\rangle_{\partial T}|
\\ \leq & \Big(\sum_{T\in {\cal T}_h} h_T^{-1}\|  Q_0\bw-Q_b\bw\|_{\partial T}^2\Big)^{\frac{1}{2}}\Big(\sum_{T\in {\cal T}_h} h_T^{-1}\|\bu_0-\bu_b\|_{\partial T}^2\Big)^{\frac{1}{2}}\\
\leq &  \Big(\sum_{T\in {\cal T}_h} h_T^{-2}\| Q_0\bw- \bw\|_{T}^2+\| Q_0\bw- \bw\|_{1, T}^2\Big)^{\frac{1}{2}} \3bar \bu_h\3bar\\
\leq &    Ch\|\bw\|_2 h^k(\|\bu\|_{k+1}+\|p\|_k).
 \end{split}
 \end{equation*} 
 
Substituting all bounds of $I_i$ ($i=1, \cdots, 6$) into \eqref{e2} and  invoking the regularity condition \eqref{regu2}, we arrive at:
\begin{equation*}
\|e_{\bu_0}\|\leq Ch^{k+1}(\|\bu\|_{k+1}+\|p\|_k).
\end{equation*}

This completes the proof. 
\end{proof}

\section{Numerical experiments}
In the 2D test,  we solve the Brinkman problem \eqref{weak} on the unit square domain $\Omega=
  (0,1)\times (0,1)$, where $\kappa=1$.
The exact solution is chosen as
\an{\label{s-2d} \b u=\p{ -8 (x^2-2 x^3 + x^4 ) ( y -3 y^2 + 2y^3 )\\
   \ \, 8 ( y -3 x^2 + 2 x^3 ) (x^2-2 x^3 + x^4 ) }, \quad  p=(x-\frac 12)^3.  
}

In the first computation, we compute the weak Galerkin finite element solutions by
  the algorithm \eqref{WG}, on triangular meshes shown in Figure \ref{f-2-1}.
We use a stabilized method where we have $r=k-1$  in \eqref{wg-k} in computing the weak
  gradient. 
The results are listed in Table \ref{t1} where we have the
  optimal order of convergence for all variables and in all norms.
  
\begin{figure}[H]
 \begin{center}\setlength\unitlength{1.0pt}
\begin{picture}(380,120)(0,0)
  \put(15,108){$G_1$:} \put(125,108){$G_2$:} \put(235,108){$G_3$:} 
  \put(0,-420){\includegraphics[width=380pt]{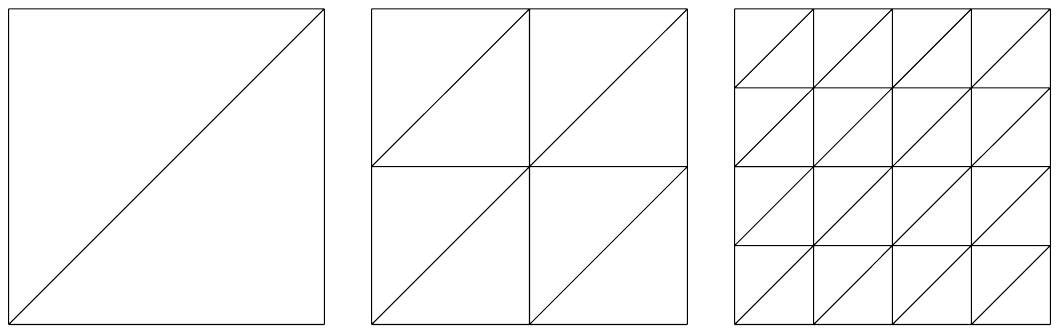}}  
 \end{picture}\end{center}
\caption{The triangular meshes for the computation in Table \ref{t1}. }\label{f-2-1}
\end{figure}

  \begin{table}[H]
  \caption{ Error profile for computing \eqref{s-2d} on meshes shown in Figure \ref{f-2-1}.} \label{t1}
\begin{center}  
   \begin{tabular}{c|rr|rr|rr}  
 \hline 
$G_i$ &  $ \|Q_h\b u - \b u_h \| $ & $O(h^r)$ &  $ \3bar Q_h \b u-\bw\3bar $ & $O(h^r)$ 
   &  $ \| p -p_h \| $ & $O(h^r)$   \\ \hline 
&\multicolumn{6}{c}{By the $P_1$-$P_1$/$P_0$ weak Galerkin finite element \eqref{V-h} and \eqref{W-h} }\\
 \hline 
 4&     0.569E-3 &  2.5&     0.119E-1 &  1.1&     0.376E-2 &  0.9 \\
 5&     0.127E-3 &  2.2&     0.580E-2 &  1.0&     0.187E-2 &  1.0 \\
 6&     0.308E-4 &  2.0&     0.288E-2 &  1.0&     0.912E-3 &  1.0 \\
 \hline 
&\multicolumn{6}{c}{By the $P_2$-$P_2$/$P_1$ weak Galerkin finite element \eqref{V-h} and \eqref{W-h} }\\
 \hline  
 4&     0.300E-4 &  3.7&     0.107E-2 &  2.2&     0.356E-3 &  2.1 \\
 5&     0.267E-5 &  3.5&     0.255E-3 &  2.1&     0.870E-4 &  2.0 \\
 6&     0.290E-6 &  3.2&     0.632E-4 &  2.0&     0.213E-4 &  2.0 \\
 \hline 
&\multicolumn{6}{c}{By the $P_3$-$P_3$/$P_2$ weak Galerkin finite element \eqref{V-h} and \eqref{W-h} }\\
 \hline  
 4&     0.209E-5 &  4.8&     0.630E-4 &  3.3&     0.269E-4 &  2.9 \\
 5&     0.807E-7 &  4.7&     0.742E-5 &  3.1&     0.334E-5 &  3.0 \\
 6&     0.387E-8 &  4.4&     0.941E-6 &  3.0&     0.401E-6 &  3.1 \\
 \hline 
&\multicolumn{6}{c}{By the $P_4$-$P_4$/$P_3$ weak Galerkin finite element \eqref{V-h} and \eqref{W-h} }\\
 \hline  
 3&     0.560E-5 &  5.8&     0.608E-4 &  4.4&     0.286E-4 &  3.8 \\
 4&     0.972E-7 &  5.8&     0.345E-5 &  4.1&     0.177E-5 &  4.0 \\
 5&     0.204E-8 &  5.6&     0.273E-6 &  3.7&     0.104E-6 &  4.1 \\
\hline 
\end{tabular} \end{center}  \end{table}

We compute again the weak Galerkin finite element solutions by
  the algorithm \eqref{WG}, but on non-convex heptagon meshes shown in Figure \ref{f-2-3}.
We use the stabilizer-free method where we take $r=k+3$  in \eqref{wg-k} in computing the weak
  gradient. Again we take $r=k-1$  in \eqref{wd-k} in computing the weak
  divergence.
The results are listed in Table \ref{t2} where we have the
  optimal order of convergence for all variables and in all norms.
  
\begin{figure}[H]
 \begin{center}\setlength\unitlength{1.0pt}
\begin{picture}(380,120)(0,0)
  \put(15,108){$G_1$:} \put(125,108){$G_2$:} \put(235,108){$G_3$:} 
  \put(0,-420){\includegraphics[width=380pt]{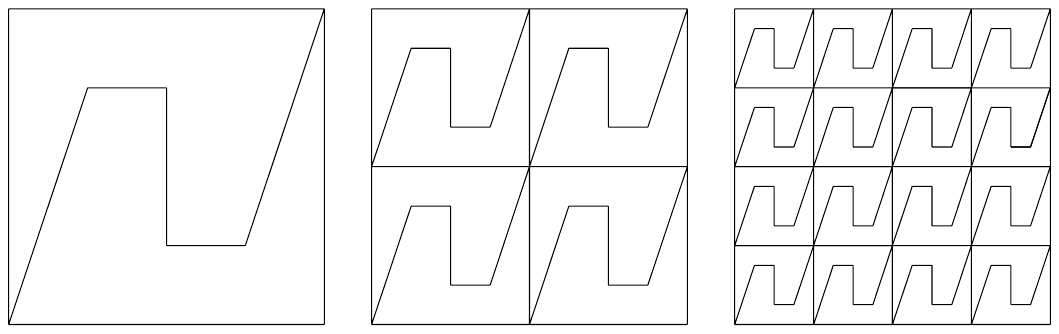}}  
 \end{picture}\end{center}
\caption{The non-convex heptagon meshes for the computation in Table \ref{t2}. }\label{f-2-3}
\end{figure}

  \begin{table}[H]
  \caption{ Error profile for computing \eqref{s-2d} on meshes shown in Figure \ref{f-2-2}.} \label{t2}
\begin{center}  
   \begin{tabular}{c|rr|rr|rr}  
 \hline 
$G_i$ &  $ \|Q_h\b u - \b u_h \| $ & $O(h^r)$ &  $ \3bar Q_h \b u-\bw\3bar $ & $O(h^r)$ 
   &  $ \| p -p_h \| $ & $O(h^r)$   \\ \hline 
&\multicolumn{6}{c}{By the $P_1$-$P_1$/$P_0$ weak Galerkin finite element \eqref{V-h} and \eqref{W-h} }\\
 \hline 
 5&     0.174E-3 &  2.1&     0.707E-2 &  1.4&     0.165E-2 &  1.1 \\
 6&     0.436E-4 &  2.0&     0.337E-2 &  1.1&     0.792E-3 &  1.1 \\
 7&     0.109E-4 &  2.0&     0.167E-2 &  1.0&     0.387E-3 &  1.0 \\
 \hline 
&\multicolumn{6}{c}{By the $P_2$-$P_2$/$P_1$ weak Galerkin finite element \eqref{V-h} and \eqref{W-h} }\\
 \hline  
 3&     0.939E-3 &  4.7&     0.401E-1 &  3.8&     0.155E-2 &  3.2 \\
 4&     0.549E-4 &  4.1&     0.297E-2 &  3.8&     0.339E-3 &  2.2 \\
 5&     0.516E-5 &  3.4&     0.373E-3 &  3.0&     0.717E-4 &  2.2 \\
 \hline 
&\multicolumn{6}{c}{By the $P_3$-$P_3$/$P_2$ weak Galerkin finite element \eqref{V-h} and \eqref{W-h} }\\
 \hline  
 1&     0.592E+0 &  0.0&     0.902E+1 &  0.0&     0.399E-1 &  0.0 \\
 2&     0.974E-2 &  5.9&     0.363E+0 &  4.6&     0.360E-2 &  3.5 \\
 3&     0.172E-3 &  5.8&     0.115E-1 &  5.0&     0.268E-3 &  3.7 \\
\hline 
\end{tabular} \end{center}  \end{table}

In the next 2D computation,  we compute the weak Galerkin finite element solutions 
    on non-convex polygon meshes shown in Figure \ref{f-2-6}. 
We get the  optimal order of convergence for all variables and in all norms in Table \ref{t3}.

\begin{figure}[H]
 \begin{center}\setlength\unitlength{1.0pt}
\begin{picture}(380,120)(0,0)
  \put(15,108){$G_1$:} \put(125,108){$G_2$:} \put(235,108){$G_3$:} 
  \put(0,-420){\includegraphics[width=380pt]{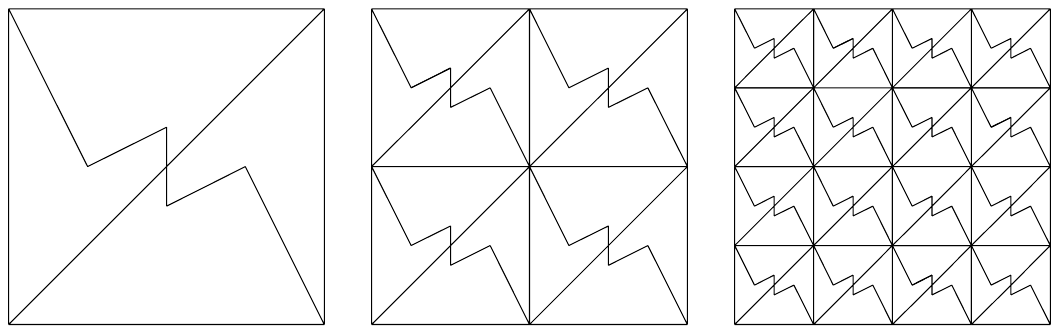}}  
 \end{picture}\end{center}
\caption{The non-convex polygon meshes for the computation in Table \ref{t3}. }\label{f-2-6}
\end{figure}

  \begin{table}[H]
  \caption{ Error profile for computing \eqref{s-2d} on meshes shown in Figure \ref{f-2-6}.} \label{t3}
\begin{center}  
   \begin{tabular}{c|rr|rr|rr}  
 \hline 
$G_i$ &  $ \|Q_h\b u - \b u_h \| $ & $O(h^r)$ &  $ \3bar Q_h \b u-\bw\3bar $ & $O(h^r)$ 
   &  $ \| p -p_h \| $ & $O(h^r)$   \\ \hline 
&\multicolumn{6}{c}{By the $P_1$-$P_1$/$P_0$ weak Galerkin finite element \eqref{V-h} and \eqref{W-h} }\\
 \hline 
 5&     0.663E-4 &  2.2&     0.423E-2 &  1.1&     0.104E-2 &  1.1 \\
 6&     0.159E-4 &  2.1&     0.210E-2 &  1.0&     0.494E-3 &  1.1 \\
 7&     0.392E-5 &  2.0&     0.105E-2 &  1.0&     0.241E-3 &  1.0 \\
 \hline 
&\multicolumn{6}{c}{By the $P_2$-$P_2$/$P_1$ weak Galerkin finite element \eqref{V-h} and \eqref{W-h} }\\
 \hline  
 4&     0.137E-4 &  3.7&     0.568E-3 &  2.3&     0.130E-3 &  2.2 \\
 5&     0.130E-5 &  3.4&     0.137E-3 &  2.1&     0.311E-4 &  2.1 \\
 6&     0.147E-6 &  3.1&     0.340E-4 &  2.0&     0.764E-5 &  2.0 \\
 \hline 
&\multicolumn{6}{c}{By the $P_3$-$P_3$/$P_2$ weak Galerkin finite element \eqref{V-h} and \eqref{W-h} }\\
 \hline  
 4&     0.137E-4 &  3.7&     0.568E-3 &  2.3&     0.130E-3 &  2.2 \\
 5&     0.130E-5 &  3.4&     0.137E-3 &  2.1&     0.311E-4 &  2.1 \\
 6&     0.147E-6 &  3.1&     0.340E-4 &  2.0&     0.764E-5 &  2.0 \\
 \hline 
&\multicolumn{6}{c}{By the $P_4$-$P_4$/$P_3$ weak Galerkin finite element \eqref{V-h} and \eqref{W-h} }\\
 \hline  
 3&     0.157E-4 &  5.0&     0.248E-3 &  4.5&     0.503E-4 &  3.7 \\
 4&     0.573E-6 &  4.8&     0.222E-4 &  3.5&     0.542E-5 &  3.2 \\
 5&     0.276E-7 &  4.4&     0.412E-5 &  2.4&     0.652E-6 &  3.1 \\
\hline 
\end{tabular} \end{center}  \end{table}

In the next 2D computation,  we compute the weak Galerkin finite element solutions 
    on non-convex polygon meshes shown in Figure \ref{f-2-7}. 
We get the  optimal order of convergence for all variables and in all norms in Table \ref{t4}.

\begin{figure}[H]
 \begin{center}\setlength\unitlength{1.0pt}
\begin{picture}(380,120)(0,0)
  \put(15,108){$G_1$:} \put(125,108){$G_2$:} \put(235,108){$G_3$:} 
  \put(0,-420){\includegraphics[width=380pt]{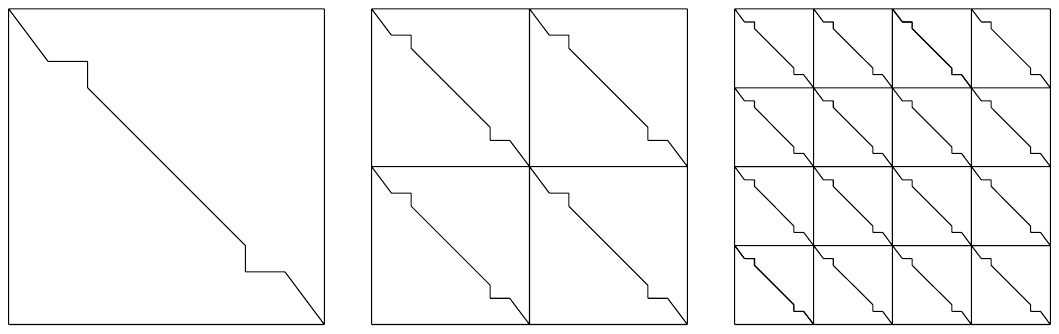}}  
 \end{picture}\end{center}
\caption{The non-convex polygon meshes for the computation in Table \ref{t4}. }\label{f-2-7}
\end{figure}

  \begin{table}[H]
  \caption{ Error profile for computing \eqref{s-2d} on meshes shown in Figure \ref{f-2-7}.} \label{t4}
\begin{center}  
   \begin{tabular}{c|rr|rr|rr}  
 \hline 
$G_i$ &  $ \|Q_h\b u - \b u_h \| $ & $O(h^r)$ &  $ \3bar Q_h \b u-\bw\3bar $ & $O(h^r)$ 
   &  $ \| p -p_h \| $ & $O(h^r)$   \\ \hline 
&\multicolumn{6}{c}{By the $P_1$-$P_1$/$P_0$ weak Galerkin finite element \eqref{V-h} and \eqref{W-h} }\\
 \hline 
 5&     0.130E-3 &  2.2&     0.592E-2 &  1.0&     0.171E-2 &  1.0 \\
 6&     0.315E-4 &  2.0&     0.294E-2 &  1.0&     0.833E-3 &  1.0 \\
 7&     0.784E-5 &  2.0&     0.147E-2 &  1.0&     0.408E-3 &  1.0 \\
 \hline 
&\multicolumn{6}{c}{By the $P_2$-$P_2$/$P_1$ weak Galerkin finite element \eqref{V-h} and \eqref{W-h} }\\
 \hline  
 5&     0.291E-5 &  3.4&     0.265E-3 &  2.1&     0.844E-4 &  2.0 \\
 6&     0.323E-6 &  3.2&     0.654E-4 &  2.0&     0.208E-4 &  2.0 \\
 7&     0.391E-7 &  3.0&     0.163E-4 &  2.0&     0.515E-5 &  2.0 \\
 \hline 
&\multicolumn{6}{c}{By the $P_3$-$P_3$/$P_2$ weak Galerkin finite element \eqref{V-h} and \eqref{W-h} }\\
 \hline  
 4&     0.217E-5 &  4.8&     0.749E-4 &  3.2&     0.293E-4 &  2.9 \\
 5&     0.918E-7 &  4.6&     0.894E-5 &  3.1&     0.364E-5 &  3.0 \\
 6&     0.496E-8 &  4.2&     0.129E-5 &  2.8&     0.441E-6 &  3.0 \\
 \hline 
&\multicolumn{6}{c}{By the $P_4$-$P_4$/$P_3$ weak Galerkin finite element \eqref{V-h} and \eqref{W-h} }\\
 \hline  
 3&     0.562E-5 &  5.8&     0.791E-4 &  4.2&     0.293E-4 &  3.9 \\
 4&     0.101E-6 &  5.8&     0.477E-5 &  4.1&     0.182E-5 &  4.0 \\
 5&     0.399E-8 &  4.7&     0.131E-5 &  1.9&     0.108E-6 &  4.1 \\
\hline 
\end{tabular} \end{center}  \end{table}

In the 3D test,  we solve the Brinkman problem \eqref{weak} on the unit cube domain $\Omega=
  (0,1)\times (0,1)\times (0,1)$, where $\kappa=1$.
The exact solution is chosen as
\an{\label{s-3d}\ad{
   \b u &=\p{ -2^{10} x^2 (1-x)^2 y^2 (1-y)^2 ( z -3 z^2 + 2z^3 )\qquad \\ \\
           \ \,2^{10} x^2 (1-x)^2 y^2 (1-y)^2 ( z -3 z^2 + 2z^3 ) \qquad \\ \\
   \ \,2^{10} \Big(( x -3 x^2 + 2x^3 )(y-y^2)^2 \qquad\qquad\qquad\\
       - (x-x^2)^2( y -3 x^2 + 2 x^3 )\Big) (z-z^2)^2 }, \\ 
    p&=10(3y^2-2y^3-y).   }
}

We first compute the weak Galerkin finite element solutions for the 3D problem \eqref{s-3d} by
  the algorithm \eqref{WG}, on tetrahedral meshes shown in Figure \ref{f-3-1}. 
The results are listed in Table \ref{t5} where we have the
  optimal order of convergence for all variables and in all norms.
  
\begin{figure}[H]
 \begin{center}\setlength\unitlength{1.0pt}
\begin{picture}(380,120)(0,0)
  \put(35,108){$G_1$:} \put(145,108){$G_2$:} \put(255,108){$G_3$:} 
  \put(0,-420){\includegraphics[width=380pt]{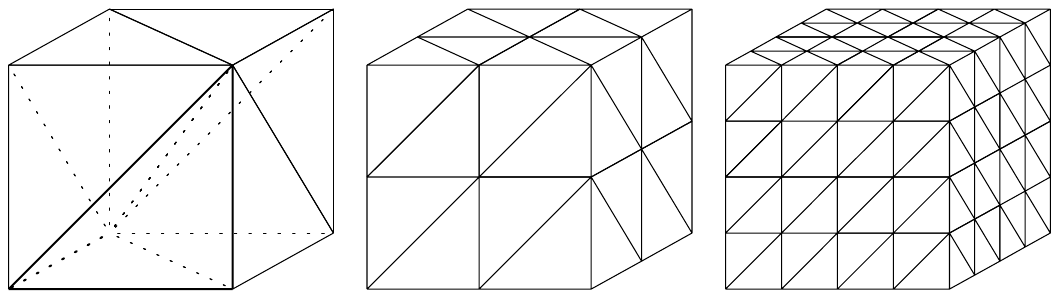}}  
 \end{picture}\end{center}
\caption{The triangular meshes for the computation in Table \ref{t5}. }\label{f-3-1}
\end{figure}

  \begin{table}[H]
  \caption{ Error profile for computing \eqref{s-3d} on meshes shown in Figure \ref{f-3-1}.} \label{t5}
\begin{center}  
   \begin{tabular}{c|rr|rr|rr}  
 \hline 
$G_i$ &  $ \|Q_h\b u - \b u_h \| $ & $O(h^r)$ &  $ \3bar Q_h \b u-\bw\3bar $ & $O(h^r)$ 
   &  $ \| p -p_h \| $ & $O(h^r)$   \\ \hline 
&\multicolumn{6}{c}{By the $P_1$-$P_1$/$P_0$ weak Galerkin finite element \eqref{V-h} and \eqref{W-h} }\\
 \hline 
 2 &    0.118E+0 &1.32 &    0.901E+0 &0.25 &    0.329E+0 &0.00 \\
 3 &    0.233E-1 &2.34 &    0.380E+0 &1.25 &    0.120E+0 &1.45 \\
 4 &    0.553E-2 &2.08 &    0.141E+0 &1.43 &    0.377E-1 &1.68 \\ 
 \hline 
&\multicolumn{6}{c}{By the $P_2$-$P_2$/$P_1$ weak Galerkin finite element \eqref{V-h} and \eqref{W-h} }\\
 \hline  
 1 &    0.214E+0 &0.00 &    0.768E+0 &0.00 &    0.728E+0 &0.00 \\
 2 &    0.332E-1 &2.69 &    0.402E+0 &0.94 &    0.192E+0 &1.92 \\
 3 &    0.493E-2 &2.75 &    0.935E-1 &2.10 &    0.341E-1 &2.50 \\
 \hline 
&\multicolumn{6}{c}{By the $P_3$-$P_3$/$P_2$ weak Galerkin finite element \eqref{V-h} and \eqref{W-h} }\\
 \hline  
 1 &    0.131E+0 &0.00 &    0.833E+0 &0.00 &    0.547E+0 &0.00 \\
 2 &    0.114E-1 &3.52 &    0.157E+0 &2.40 &    0.859E-1 &2.67 \\
 3 &    0.352E-2 &1.70 &    0.317E-1 &2.31 &    0.103E-1 &3.06 \\
\hline 
\end{tabular} \end{center}  \end{table}

We next compute the weak Galerkin finite element solutions for \eqref{s-3d} on a
 family of non-convex polyhedral meshes shown in Figure \ref{f-3-2},
where the six pyramids are combined into three non-convex polyhedra inside each cube.
The results are listed in Table \ref{t6} where we obtain the
  optimal order of convergence.
  
\begin{figure}[H]
 \begin{center}\setlength\unitlength{1.0pt}
\begin{picture}(380,120)(0,0)
  \put(35,108){$G_1$:} \put(145,108){$G_2$:} \put(255,108){$G_3$:} 
  \put(0,-420){\includegraphics[width=380pt]{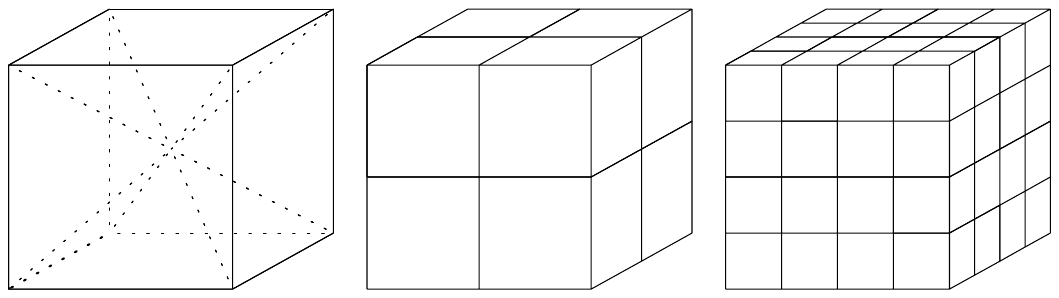}}  
 \end{picture}\end{center}
\caption{The triangular meshes for the computation in Table \ref{t6}. }\label{f-3-2}
\end{figure}

  \begin{table}[H]
  \caption{ Error profile for computing \eqref{s-3d} on meshes shown in Figure \ref{f-3-2}.} \label{t6}
\begin{center}  
   \begin{tabular}{c|rr|rr|rr}  
 \hline 
$G_i$ &  $ \|Q_h\b u - \b u_h \| $ & $O(h^r)$ &  $ \3bar Q_h \b u-\bw\3bar $ & $O(h^r)$ 
   &  $ \| p -p_h \| $ & $O(h^r)$   \\ \hline 
&\multicolumn{6}{c}{By the $P_1$-$P_1$/$P_0$ weak Galerkin finite element \eqref{V-h} and \eqref{W-h} }\\
 \hline 
 2 &    0.849E-1 &1.28 &    0.396E+1 &0.00 &    0.446E+0 &0.00 \\
 3 &    0.384E-1 &1.14 &    0.167E+1 &1.25 &    0.236E+0 &0.92 \\
 4 &    0.994E-2 &1.95 &    0.677E+0 &1.30 &    0.863E-1 &1.45 \\
 \hline 
&\multicolumn{6}{c}{By the $P_2$-$P_2$/$P_1$ weak Galerkin finite element \eqref{V-h} and \eqref{W-h} }\\
 \hline  
 1 &    0.241E+0 &0.00 &    0.341E+1 &0.00 &    0.171E+1 &0.00 \\
 2 &    0.918E-1 &1.39 &    0.370E+1 &0.00 &    0.956E+0 &0.84 \\
 3 &    0.861E-2 &3.41 &    0.816E+0 &2.18 &    0.189E+0 &2.34 \\
 \hline 
&\multicolumn{6}{c}{By the $P_3$-$P_3$/$P_2$ weak Galerkin finite element \eqref{V-h} and \eqref{W-h} }\\
 \hline  
 1 &    0.188E+0 &0.00 &    0.457E+1 &0.00 &    0.258E+1 &0.00 \\
 2 &    0.860E-1 &1.13 &    0.511E+1 &0.00 &    0.154E+1 &0.75 \\
 3 &    0.852E-2 &3.34 &    0.376E+0 &3.77 &    0.182E+0 &3.08 \\
\hline 
\end{tabular} \end{center}  \end{table}

  We compute the weak Galerkin finite element solutions for \eqref{s-3d} 
   on non-convex polyhedral meshes shown in Figure \ref{f-3-4}, in Table \ref{t7}.
We use the stabilizer-free method where we take $r=k+3$  in \eqref{wg-k} in computing the weak
  gradient.  
  
\begin{figure}[H]
 \begin{center}\setlength\unitlength{1.0pt}
\begin{picture}(380,120)(0,0)
  \put(35,108){$G_1$:} \put(145,108){$G_2$:} \put(255,108){$G_3$:} 
  \put(0,-420){\includegraphics[width=380pt]{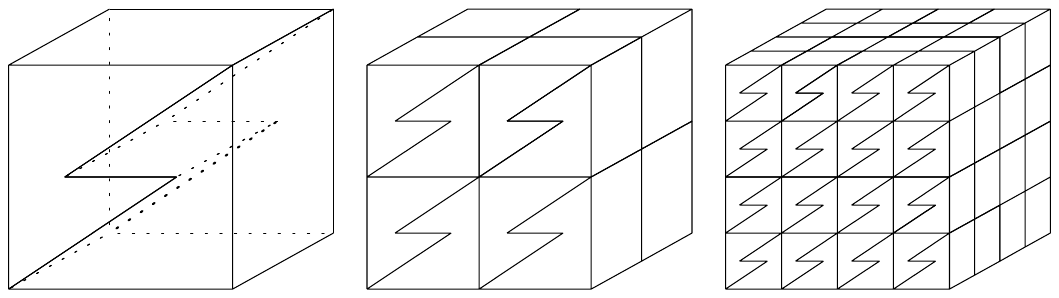}}  
 \end{picture}\end{center}
\caption{The triangular meshes for the computation in Table \ref{t7}. }\label{f-3-4}
\end{figure}

  \begin{table}[H]
  \caption{ Error profile for computing \eqref{s-3d} on meshes shown in Figure \ref{f-3-4}.} \label{t7}
\begin{center}  
   \begin{tabular}{c|rr|rr|rr}  
 \hline 
$G_i$ &  $ \|Q_h\b u - \b u_h \| $ & $O(h^r)$ &  $ \3bar Q_h \b u-\bw\3bar $ & $O(h^r)$ 
   &  $ \| p -p_h \| $ & $O(h^r)$   \\ \hline 
&\multicolumn{6}{c}{By the $P_1$-$P_1$/$P_0$ weak Galerkin finite element \eqref{V-h} and \eqref{W-h} }\\
 \hline 
 2 &    0.439E+0 &1.94 &    0.285E+1 &0.04 &    0.567E+0 &0.06 \\
 3 &    0.468E-1 &3.23 &    0.935E+0 &1.61 &    0.241E+0 &1.24 \\
 4 &    0.759E-2 &2.62 &    0.284E+0 &1.72 &    0.656E-1 &1.88 \\
 \hline 
&\multicolumn{6}{c}{By the $P_2$-$P_2$/$P_1$ weak Galerkin finite element \eqref{V-h} and \eqref{W-h} }\\
 \hline  
 2 &    0.215E+0 &4.18 &    0.396E+1 &1.58 &    0.168E+1 &1.79 \\
 3 &    0.159E-1 &3.76 &    0.746E+0 &2.41 &    0.252E+0 &2.74 \\
 4 &    0.365E-2 &2.12 &    0.104E+0 &2.85 &    0.339E-1 &2.90 \\ 
 \hline 
&\multicolumn{6}{c}{By the $P_3$-$P_3$/$P_2$ weak Galerkin finite element \eqref{V-h} and \eqref{W-h} }\\
 \hline  
 1 &    0.404E+1 &0.00 &    0.170E+2 &0.00 &    0.939E+1 &0.00 \\
 2 &    0.176E+0 &4.52 &    0.547E+1 &1.63 &    0.212E+1 &2.15 \\
 3 &    0.769E-2 &4.52 &    0.441E+0 &3.63 &    0.173E+0 &3.61 \\
\hline 
\end{tabular} \end{center}  \end{table}

As the last test, compute the weak Galerkin finite element solutions for \eqref{s-3d} 
   on non-convex polyhedral meshes shown in Figure \ref{f-3-7}, in Table \ref{t9}. 
 We get the optimal order convergence in all cases.
 
\begin{figure}[H]
 \begin{center}\setlength\unitlength{1.0pt}
\begin{picture}(380,120)(0,0)
  \put(35,108){$G_1$:} \put(145,108){$G_2$:} \put(255,108){$G_3$:} 
  \put(0,-420){\includegraphics[width=380pt]{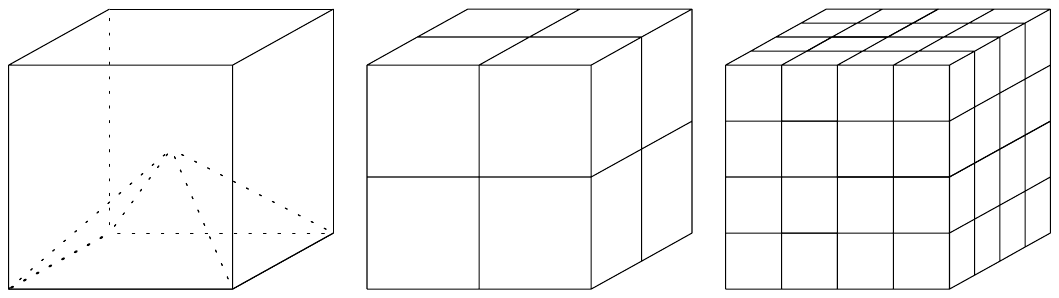}}  
 \end{picture}\end{center}
\caption{The triangular meshes for the computation in Table \ref{t9}. }\label{f-3-7}
\end{figure}

  \begin{table}[H]
  \caption{ Error profile for computing \eqref{s-3d} on meshes shown in Figure \ref{f-3-7}.} \label{t9}
\begin{center}  
   \begin{tabular}{c|rr|rr|rr}  
 \hline 
$G_i$ &  $ \|Q_h\b u - \b u_h \| $ & $O(h^r)$ &  $ \3bar Q_h \b u-\bw\3bar $ & $O(h^r)$ 
   &  $ \| p -p_h \| $ & $O(h^r)$   \\ \hline 
&\multicolumn{6}{c}{By the $P_1$-$P_1$/$P_0$ weak Galerkin finite element \eqref{V-h} and \eqref{W-h} }\\
 \hline 
 2 &    0.588E+0 &1.55 &    0.356E+1 &0.00 &    0.135E+1 &0.00 \\
 3 &    0.682E-1 &3.11 &    0.140E+1 &1.35 &    0.485E+0 &1.48 \\
 4 &    0.106E-1 &2.69 &    0.443E+0 &1.66 &    0.123E+0 &1.98 \\
 \hline 
&\multicolumn{6}{c}{By the $P_2$-$P_2$/$P_1$ weak Galerkin finite element \eqref{V-h} and \eqref{W-h} }\\
 \hline  
 2 &    0.351E+0 &4.38 &    0.431E+1 &2.05 &    0.263E+1 &2.68 \\
 3 &    0.252E-1 &3.80 &    0.125E+1 &1.79 &    0.491E+0 &2.42 \\
 4 &    0.373E-2 &2.76 &    0.182E+0 &2.78 &    0.696E-1 &2.82 \\
 \hline 
&\multicolumn{6}{c}{By the $P_3$-$P_3$/$P_2$ weak Galerkin finite element \eqref{V-h} and \eqref{W-h} }\\
 \hline  
 1 &    0.658E+1 &0.00 &    0.317E+2 &0.00 &    0.437E+2 &0.00 \\
 2 &    0.267E+0 &4.62 &    0.899E+1 &1.82 &    0.445E+1 &3.29 \\
 3 &    0.169E-1 &3.99 &    0.855E+0 &3.39 &    0.428E+0 &3.38 \\
\hline 
\end{tabular} \end{center}  \end{table}


\begin{thebibliography}{99}

 \bibitem{b1} {\sc S. Badia and R. Codina},  {\em Unified stabilized finite element formulations for the Stokes and the
Darcy problems}, SIAM J. Numer. Anal., 47 (2009), 1971-2000.

    \bibitem{11} {\sc J. Knn and R. Stenberg},  {\em H(div)-conforming finite elements for the Brinkman problem}, Math. Models and Meth. Applied Sciences, 11 (2011), 2227-2248.

 \bibitem{12} {\sc K. Mardal, X. Tai, and R. Winther},  {\em  A Robust finite element method for Darcy-Stokes 
flow},
SIAM J. Numer. Anal., 40 (2002), 1605-1631.

   
 

   \bibitem{wg11}{\sc  S. Cao, C. Wang and J. Wang},  {\em A new numerical method for div-curl Systems with Low Regularity Assumptions}, Computers and Mathematics with Applications, vol. 144, pp. 47-59, 2022.
 


  
 \bibitem{wg14}{\sc D. Li, Y. Nie, and C. Wang},  {\em Superconvergence of Numerical Gradient for Weak Galerkin Finite Element Methods on Nonuniform Cartesian Partitions in Three Dimensions}, Computers and Mathematics with Applications, vol 78(3), pp. 905-928, 2019.  
  \bibitem{wg1} {\sc D. Li, C. Wang and J. Wang},  {\em An Extension of the Morley Element on General Polytopal Partitions Using Weak Galerkin Methods}, Journal of Scientific Computing, 100, vol 27, 2024.  
 \bibitem{wg2} {\sc D. Li, C. Wang and S. Zhang},  {\em Weak Galerkin methods for elliptic interface problems on curved polygonal partitions}, Journal of Computational and Applied Mathematics, pp. 115995, 2024. 
\bibitem{wg5} {\sc D. Li, C. Wang, J.  Wang and X. Ye},  {\em Generalized weak Galerkin finite element methods for second order elliptic problems}, Journal of Computational and Applied Mathematics, vol. 445, pp. 115833, 2024.
 \bibitem{wg6} {\sc D. Li, C. Wang, J. Wang and S. Zhang},  {\em High Order Morley Elements for Biharmonic Equations on Polytopal Partitions}, Journal of Computational and Applied Mathematics, Vol. 443, pp. 115757, 2024.
 \bibitem{wg7} {\sc D. Li, C. Wang and J. Wang},  {\em Curved Elements in Weak Galerkin Finite Element Methods}, Computers and Mathematics with Applications, Vol. 153, pp. 20-32, 2024.
\bibitem{wg8} {\sc D. Li, C. Wang and J. Wang},  {\em Generalized Weak Galerkin Finite Element Methods for Biharmonic Equations}, Journal of Computational and Applied Mathematics, vol. 434, 115353, 2023.
 
  \bibitem{wg13}{\sc  D. Li, C. Wang, and J. Wang},  {\em Superconvergence of the Gradient Approximation for Weak Galerkin Finite Element Methods on Rectangular Partitions}, Applied Numerical Mathematics, vol. 150, pp. 396-417, 2020.
 

 

 \bibitem{fedi}{\sc  C. Wang},  {\em A Preconditioner for the FETI-DP Method for Mortar-Type Crouzeix-Raviart Element Discretization}, Applications of Mathematics, Vol. 59, 6, pp. 653-672, 2014. 
   \bibitem{wg15}{\sc C. Wang},  {\em New Discretization Schemes for Time-Harmonic Maxwell Equations by Weak Galerkin Finite Element Methods}, Journal of Computational and Applied Mathematics, Vol. 341, pp. 127-143, 2018.  

    
 \bibitem{wg17}{\sc C. Wang and J. Wang},  {\em Discretization of Div-Curl Systems by Weak Galerkin Finite Element Methods on Polyhedral Partitions}, Journal of Scientific Computing, Vol. 68, pp. 1144-1171, 2016.    
   \bibitem{wg19}{\sc C. Wang and J. Wang},  {\em A Hybridized Formulation for Weak Galerkin Finite Element Methods for Biharmonic Equation on Polygonal or Polyhedral Meshes}, International Journal of Numerical Analysis and Modeling, Vol. 12, pp. 302-317, 2015. 
 \bibitem{wg20}{\sc  J. Wang and C. Wang},  {\em Weak Galerkin Finite Element Methods for Elliptic PDEs}, Science China, Vol. 45, pp. 1061-1092, 2015.  
 \bibitem{wg21}{\sc C. Wang and J. Wang},  {\em An Efficient Numerical Scheme for the Biharmonic Equation by Weak Galerkin Finite Element Methods on Polygonal or Polyhedral Meshes}, Journal of Computers and Mathematics with Applications, Vol. 68, 12, pp. 2314-2330, 2014.  
 
   \bibitem{wg18}{\sc C. Wang, J. Wang, R. Wang and R. Zhang},  {\em A Locking-Free Weak Galerkin Finite Element Method for Elasticity Problems in the Primal Formulation}, Journal of Computational and Applied Mathematics, Vol. 307, pp. 346-366, 2016.   
 

 \bibitem{wg12}{\sc  C. Wang, J. Wang, X. Ye and S. Zhang},  {\em De Rham Complexes for Weak Galerkin Finite Element Spaces}, Journal of Computational and Applied Mathematics, vol. 397, pp. 113645, 2021.
 
 \bibitem{wg3} {\sc C. Wang, J. Wang and S. Zhang},  {\em Weak Galerkin Finite Element Methods for Optimal Control Problems Governed by Second Order Elliptic Partial Differential Equations}, Journal of Computational and Applied Mathematics, in press, 2024. 
 
 \bibitem{itera} {\sc C. Wang, J. Wang and S. Zhang},  {\em A parallel iterative procedure for weak Galerkin methods for second order elliptic problems}, International Journal of Numerical Analysis and Modeling, vol. 21(1), pp. 1-19, 2023.
 \bibitem{wg9} {\sc C. Wang, J. Wang and S. Zhang},  {\em Weak Galerkin Finite Element Methods for Quad-Curl Problems}, Journal of Computational and Applied Mathematics, vol. 428, pp. 115186, 2023.
   \bibitem{wy3655} {\sc J. Wang, and X. Ye}, {\em A weak Galerkin mixed finite element method for second-order elliptic problems}, Math. Comp., vol. 83, pp. 2101-2126, 2014.


  \bibitem{wg4} {\sc C. Wang, X. Ye and S. Zhang},  {\em A Modified weak Galerkin finite element method for the Maxwell equations on polyhedral meshes}, Journal of Computational and Applied Mathematics, vol. 448, pp. 115918, 2024. 
 \bibitem{wg10}{\sc  C. Wang and S. Zhang},  {\em A Weak Galerkin Method for Elasticity Interface Problems}, Journal of Computational and Applied Mathematics, vol. 419, 114726, 2023.

   \bibitem{autostokes}{\sc C. Wang and S. Zhang},  {\em Auto-Stabilized Weak Galerkin Finite Element Methods for Stokes Equations on Non-Convex Polytopal Meshes}, Journal of Computational Physics, vol. 533, 114006, 2025. 

  
 
 \bibitem{wg16}{\sc  C. Wang and H. Zhou},  {\em A Weak Galerkin Finite Element Method for a Type of Fourth Order Problem arising from Fluorescence Tomography}, Journal of Scientific Computing, Vol. 71(3), pp. 897-918, 2017.  

 \end{thebibliography}
\end{document}